\documentclass[11pt,a4paper]{amsart}

\usepackage[T1]{fontenc}
\usepackage{lmodern}
\usepackage[utf8]{inputenc}
\usepackage[english]{babel}
\usepackage[a4paper,hmargin=2.4cm,vmargin=2cm]{geometry}

\usepackage{tikz}
\usetikzlibrary{arrows.meta} 
\usetikzlibrary{calc}

\usepackage[backend=biber,
            style=numeric,
            sorting=nyt]{biblatex}
\usepackage{microtype}

\usepackage{amsmath,amssymb,amsthm,mathtools}
\usepackage{enumitem}

\usepackage{xcolor}
\definecolor{LinkBlue}{HTML}{1F3A5F}
\usepackage[
    colorlinks=true,
    linkcolor=LinkBlue,
    citecolor=LinkBlue,
    urlcolor=LinkBlue,
    hypertexnames=false
]{hyperref}

\usepackage{fancyhdr}
\fancypagestyle{plain}{%
  \fancyhf{}%
}



\theoremstyle{plain}
\newtheorem{theorem}{Theorem}[section]
\newtheorem{lemma}[theorem]{Lemma}
\newtheorem{proposition}[theorem]{Proposition}
\newtheorem{corollary}[theorem]{Corollary}

\theoremstyle{definition}

\newtheoremstyle{remarkstyle}
  {1.2\topsep}{1.2\topsep}   
  {\normalfont\small}        
  {}                         
  {\itshape}                 
  {.}{.5em}{}

\theoremstyle{remarkstyle}
\newtheorem{remark}[theorem]{Remark}

\theoremstyle{plain}
\newtheorem{theoremA}{Theorem}

\newtheorem{theoremB}{Theorem}

\newtheorem{theoremC}{Theorem}

\newtheorem{theoremD}{Theorem}

\makeatletter
\def\orig@section{\@startsection{section}{1}%
  \z@{.7\linespacing\@plus\linespacing}{.5\linespacing}%
  {\normalfont\bfseries\centering}}

\def\section{%
  \@ifstar{\uc@sect@star}{\uc@sect@nostar}}

\newcommand{\uc@sect@star}[1]{%
  \orig@section*{\MakeUppercase{#1}}}

\newcommand{\uc@sect@nostar}{%
  \@ifnextchar[{\uc@sect@opt}{\uc@sect@plain}}

\def\uc@sect@opt[#1]#2{%
  \orig@section[#1]{\MakeUppercase{#2}}}

\newcommand{\uc@sect@plain}[1]{%
  \orig@section[#1]{\MakeUppercase{#1}}}
\makeatother

\newcommand{\R}{\mathbb{R}}

\newcommand{\HH}{\mathbb{H}}
\newcommand{\Sph}{\mathbb{S}}

\DeclareMathOperator{\Ric}{Ric}

\title{Congruence theorems for capillary hypersurfaces and a substatic
Heintze--Karcher inequality}

\author[F. Guimar\~aes]{Felippe Guimar\~aes}
\address{Departamento de Matem\'atica,
         Universidade Federal do Rio de Janeiro,
         Av.\ Athos da Silveira Ramos, 149, Bloco C,
         21941-909 Rio de Janeiro, Brazil}
\email{felippe@im.ufrj.br}


\date{\today}

\begin{document}

\begin{abstract}
We prove congruence theorems for capillary hypersurfaces of Euclidean space supported on a hyperplane or on a sphere, extending a result of Ros for closed hypersurfaces. We also establish a Heintze--Karcher type inequality for capillary hypersurfaces in substatic manifolds.
\end{abstract}

\maketitle
\thispagestyle{empty}

\section{Introduction}

An interesting problem in submanifold theory is to study the uniqueness of
an isometric immersion $\psi\colon \Sigma^n\to\R^{n+p}$ of a Riemannian manifold,
that is, when $\psi$ is unique up to isometries of the ambient space, in which
case the two immersions are said to be congruent. We say that $\psi$ is
\emph{isometrically rigid} if every other isometric immersion of $\Sigma^n$ into
$\R^{n+p}$ is congruent to it. Locally, the non-degeneracy of the second fundamental form gives information
about rigidity. A classical result in this direction, stated by Beez and proved
by Killing~\cite{KillingRigidity}, asserts that a hypersurface whose second
fundamental form has rank at least three at every point is rigid (we refer to
Chapter~4 in~\cite{DajczerTojeiroBook} for a discussion). In the global setting, additional hypotheses are usually needed. For instance, Sacksteder~\cite{SackstederRigidity}
proved that a compact hypersurface of Euclidean space is isometrically
rigid, provided its totally geodesic points do not disconnect it (see also \cite{DajczerGromoll85b,FloritGuimaraes20}). We point out that some of the techniques behind these
results, more precisely the flat bilinear forms introduced by
Moore~\cite{Moore77}, require the hypersurface to have dimension at least three
and the isometric immersion to be of relatively low codimension, as can be seen
from the \emph{Main Lemma} in~\cite{DajczerTojeiroBook}.

A global rigidity result of a different nature is the following classical
statement about curves due to Schur~\cite{Schur21}, which is also one of the motivations for this work.

\begin{theorem}[\cite{Schur21}]\label{thm:schur}
Let $\gamma$ be a smooth closed convex plane curve and let $\tilde\gamma$
be a smooth closed space curve of the same length, both parametrised by
arclength. If the curvature of $\tilde\gamma$ is at most the curvature of
$\gamma$ at corresponding values of arclength, then the two curves are
congruent.
\end{theorem}

An extension of Theorem~\ref{thm:schur} to hypersurfaces was obtained by
Ros~\cite{Ros88}. There the curvature of a curve is replaced by the mean curvature
of a hypersurface, the hypothesis on the length by the requirement that both
immersions induce the same metric, and the proof relies on integral inequalities
derived from Reilly's formula. The convexity of the comparison
object becomes mean convexity, and the conclusion is actually stronger than isometric
rigidity, since the codimension of the second immersion is arbitrary and the
theorem also provides a reduction of the codimension.

\begin{theorem}[Theorem 2 in~\cite{Ros88}]\label{thm:ros}
Let $\psi\colon \Sigma^n\to\R^{n+1}$ be a compact embedded hypersurface whose mean
curvature $H_\Sigma$, with respect to the outward normal, is nonnegative. If
$\psi'\colon \Sigma^n\to\R^{n+p}$, $p\geq 1$, is an isometric immersion with mean
curvature vector $\vec H'$ satisfying $|\vec H'|\leq H_\Sigma$ everywhere, then
$\psi'=\iota\circ\psi$ for a totally geodesic inclusion
$\iota\colon\R^{n+1}\to\R^{n+p}$. In particular $\psi$ and $\psi'$ are
congruent.
\end{theorem}

The same scheme was carried out in the hemisphere by Freitas and
Guimar\~aes~\cite{FreitasGuimaraes25}, where a codimension reduction and
congruence theorem is obtained for compact hypersurfaces of $\Sph^{n+1}_+$
admitting a mean convex isometric embedding, using the generalization of
Reilly's formula for space forms due to Qiu and Xia~\cite{QiuXia15}.

In all of the above the manifold $\Sigma^n$ is closed. Our
first purpose is to address the corresponding question about global rigidity of
mean convex hypersurfaces when $\Sigma^n$ has nonempty boundary lying on a fixed
umbilical support hypersurface of the Euclidean space, meeting it at a constant
contact angle $\theta$. Such immersions are called $\theta$-capillary, and they
arise as critical points of an energy functional in the presence of the support
(see~\cite{Finn86,RosSouam97}), the case $\theta=\frac{\pi}{2}$ being known as free
boundary. We write $\Gamma=\partial\Sigma$ and identify an embedded
hypersurface with its image.

Our first result extends Theorem~\ref{thm:ros} to capillary hypersurfaces
supported on a hyperplane, without restriction on the codimension of the
second immersion.

\begin{theoremA}\label{thmA}
Let $\psi\colon \Sigma^n\to\R^{n+1}$ be a compact embedded hypersurface,
$\theta$-capillary in a hyperplane $S$ with $\theta\in(0,\frac{\pi}{2}]$ and
$\Sigma\cap S=\Gamma$, with mean curvature $H_\Sigma\geq 0$. Let
$\psi'\colon \Sigma^n\to\R^{n+p}$ be an isometric immersion with mean curvature
vector $\vec H'$ satisfying $|\vec H'|\leq H_\Sigma$, such that
$\psi'(\Gamma)\subset\hat S$ for some hyperplane $\hat S\subset\R^{n+p}$
and $\psi'$ is $\theta$-capillary in $\hat S$ with the same angle $\theta$.
Then $\psi'$ differs from $\psi$ by a rigid motion.
\end{theoremA}

Observe that there are no compact minimal hypersurfaces satisfying the hypotheses of
the theorem above. For a spherical support, free boundary minimal hypersurfaces have attracted considerable interest (see~\cite{FraserSchoen2011,LiFreeBoundarySurvey2020} and references therein), and adapting the techniques of \cite{Ros88} we obtain isometric rigidity in this class.

\begin{theoremB}\label{thmB}
Let $\psi\colon \Sigma^n\to\R^{n+1}$ be a compact embedded minimal hypersurface with
$\psi(\Sigma)\subset\overline{\mathbb{B}}^{\,n+1}$ and free boundary in $\Sph^n$,
and set $\Gamma=\partial\Sigma$. Assume that each connected component of
$\Gamma$ bounds a mean convex domain of $\Sph^n$ contained in an open
hemisphere, and that these domains are pairwise disjoint. If
$\psi'\colon \Sigma^n\to\R^{n+1}$ is a minimal isometric immersion with free
boundary in $\Sph^n$, then $\psi'$ differs from $\psi$ by a rigid motion.
\end{theoremB}

In the capillary case, that is, $0<\theta<\pi/2$, we also obtain the corresponding result under a mean
curvature comparison.

\begin{theoremC}\label{thmC}
Let $\psi\colon \Sigma^n\to\R^{n+1}$ be a compact embedded hypersurface with
$\psi(\Sigma)\subset\overline{\mathbb{B}}^{\,n+1}$ and mean curvature $H_\Sigma\geq 0$,
$\theta$-capillary in $\Sph^n$ with $\theta\in(0,\frac{\pi}{2})$, and set
$\Gamma=\partial\Sigma$. Assume that each connected component of $\Gamma$
bounds a mean convex domain of $\Sph^n$ contained in an open hemisphere, and
that these domains are pairwise disjoint. If $\psi'\colon \Sigma^n\to\R^{n+1}$ is an
isometric immersion with mean curvature vector $\vec H'$ satisfying
$|\vec H'|\leq H_\Sigma$, $\theta$-capillary in $\Sph^n$ with the same angle $\theta$,
then $\psi'$ differs from $\psi$ by a rigid motion.
\end{theoremC}

As in~\cite{Ros88}, the proofs make use of Reilly's formula, which we apply
after extending the coordinates of the isometric immersion to the region enclosed by $\psi(\Sigma)$ and
the support. The difference is that the boundary of this region is no longer smooth. It
consists of hypersurfaces meeting along the corner $\Gamma$, and this is why we need information about how $\Gamma$ lies in the supporting sphere.

The second part is also motivated by the work of Ros~\cite{Ros88}, who
used the same tools to obtain a Heintze--Karcher inequality and an
Alexandrov-type result. In the capillary setting this topic has been studied extensively. In particular, the work of Jia, Xia and
Zhang~\cite{JiaXiaZhang2023}, based on Reilly-type formulas and mixed
boundary value problems, is the starting point for this part of our work
(see also \cite{JiaWangXiaZhangWedge2025,JiaWangXiaZhangAnisotropic2023,JiaLuXiaZhang2025}
for related results on wedges, anisotropic capillary hypersurfaces and
partially overdetermined problems). We also mention the result of
Hu, Wei, Xia and Zhou~\cite{HuWeiXiaZhou2025} in a hyperbolic half-space,
whose proof uses a geodesic normal flow for a Randers metric, obtains a stronger inequality, and does not require \(0 < \theta < \frac{\pi}{2}\).

Our purpose here is to obtain a Heintze--Karcher inequality for capillary hypersurfaces in a spherical
cap, following closely the method of~\cite{JiaXiaZhang2023}, but without
fixing the ambient model or choosing conformal Killing fields. For this, we use the weighted Reilly formula obtained by Li and Xia
\cite{LiXiaReillySubstatic}. Its natural ambient setting is that of a
\emph{substatic triple} $(\overline M^{n+1},g,V)$, that is, one for which the tensor
\[
  \mathcal Q_V:=(\bar\Delta V)g-\bar\nabla^2V+V\,\Ric
\]
is nonnegative. This framework contains the static space forms and the warped
products considered by Brendle \cite{Brendle2013}. In the
presence of a supporting hypersurface $T$, the weighted Reilly formula singles out the
boundary tensor
\[
  h^T-\frac{V_N}{V}\,g_T,
\]
and this leads to the corresponding notion of $V$-convexity of the support.

The capillary correction term arises directly from the
boundary contribution of Reilly's formula. In particular, no Killing or
conformal Killing field is needed in the proof of the general inequality. This allows us to recover the corresponding space form inequalities in a unified way.

\begin{theoremD}\label{thmD}
Let $(\overline M^{n+1},g,V)$ be substatic on a compact domain
$\Omega$ with $V>0$ on $\overline\Omega$ and boundary
$\partial\Omega=\Sigma\cup T$ formed by two smooth hypersurfaces meeting
along their common smooth corner $\Gamma=\partial\Sigma=\partial T$. Assume that $\Sigma$ is compact,
embedded, has mean curvature $H_\Sigma>0$, and meets $T$ at a
constant angle $\theta\in(0,\frac{\pi}{2})$. If
\[
  h^T-\frac{V_N}{V}g_T\geq0
\]
then
\[
  \mathcal K_\theta
  :=\tan\theta\int_\Gamma V\,ds-\int_T H_TV\,dA_T
\]
is positive and
\[
  \int_\Sigma\frac{V}{H_\Sigma}\,dA
  \geq
  \frac{n+1}{n}\int_\Omega V\,d\Omega
  +\frac{\left(\int_TV\,dA_T\right)^2}{\mathcal K_\theta}.
\]
If equality holds, then $\Sigma$ is totally umbilical.
\end{theoremD}

The paper is organized as follows. In Section~\ref{sec:preliminaries} we fix
the capillary notation, recall the classical and weighted Reilly formulas, and
introduce the substatic and boundary tensors that will be used later. In
Section~\ref{sec:congruence} we prove the congruence results. In
Section~\ref{sec:substatic-HK} we establish the capillary Heintze--Karcher
inequality in a substatic manifold and obtain an Alexandrov-type result.

\section{Preliminaries}\label{sec:preliminaries}

We fix the notation for capillary domains and recall the Reilly identities
used in the proofs. Let $(\overline M^{n+1},g)$ be a Riemannian manifold and let
$S\subset\overline M$ be a smooth, orientable, embedded hypersurface, called
the \emph{support}. Let $\Sigma^n\subset\overline M$ be a compact embedded
hypersurface with boundary $\Gamma:=\partial\Sigma\subset S$. Together with a
region $T\subset S$, the hypersurface $\Sigma$ bounds a domain
$\Omega\subset\overline M$ with $\partial\Omega=\Sigma\cup T$.
We assume that $\Omega$ is connected and relatively compact. Along the
corner $\Gamma=\partial\Sigma=\partial T$, we use the following unit normals
\begin{itemize}
  \item $\nu$, the unit normal to $\Sigma$ pointing out of $\Omega$,
  \item $\mu$, the outward conormal of $\Gamma^{n-1}$ in $\Sigma$,
  \item $N$, the unit normal to $S$ along $T$, pointing out of $\Omega$,
  \item $\bar\nu$, the outward conormal of $\Gamma^{n-1}$ in $T$.
\end{itemize}
The pairs $\{\nu,\mu\}$ and $\{N,\bar\nu\}$ are two orthonormal bases of the
normal plane of $\Gamma$ in $\overline M$. The hypersurface $\Sigma$ is
\emph{$\theta$-capillary} in $S$, with $\theta\in(0,\pi)$, if $g(\nu,N)=-\cos\theta$ along $\Gamma$. With these conventions,
\begin{equation}\label{eq:capillarity}
  \mu = \sin\theta\,N+\cos\theta\,\bar\nu,\qquad
  \nu = -\cos\theta\,N+\sin\theta\,\bar\nu.
\end{equation}
When $\theta=\pi/2$ the hypersurface is called \emph{free boundary}.

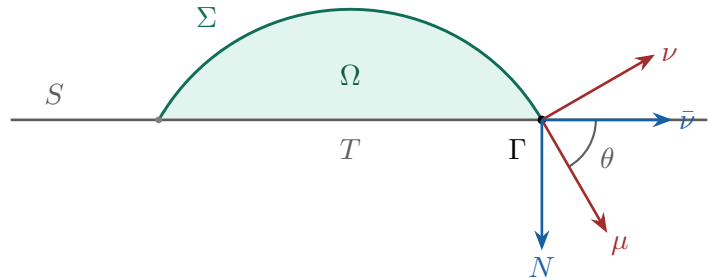
\begin{figure}[ht]
\centering
\definecolor{suppcol}{RGB}{95,94,90}
\definecolor{surfcol}{RGB}{15,110,86}
\definecolor{domcol}{RGB}{225,245,238}
\definecolor{framea}{RGB}{163,45,45}
\definecolor{frameb}{RGB}{24,95,165}
\begin{tikzpicture}[>=Stealth, scale=1.15,
  vec/.style={line width=1pt,->}]

  \coordinate (C) at (0,-1.270);
  \coordinate (L) at (-2.2,0);
  \coordinate (R) at (2.2,0);

  \fill[domcol] ($(C)+(30:2.54)$) arc[start angle=30,end angle=150,radius=2.54] -- cycle;
  \draw[suppcol,line width=1pt] (-3.9,0) -- (4.1,0);
  \draw[surfcol,line width=1.1pt] ($(C)+(150:2.54)$) arc[start angle=150,end angle=30,radius=2.54];

  \node[suppcol] at (-3.4,0.30) {$S$};
  \node[surfcol] at (-1.65,1.18) {$\Sigma$};
  \node[suppcol] at (0,-0.33) {$T$};
  \node[surfcol!85!black] at (0,0.52) {$\Omega$};

  \fill (R) circle (1.4pt);
  \fill[black!50] (L) circle (1.1pt);
  \node at (1.92,-0.33) {$\Gamma$};

  \draw[suppcol,line width=0.8pt] ($(R)+(0:0.62)$) arc[start angle=0,end angle=-60,radius=0.62];
  \node[suppcol] at ($(R)+(-30:0.87)$) {$\theta$};

  \draw[framea,vec] (R) -- ++(1.299,0.75)  node[right,inner sep=2pt]{$\nu$};
  \draw[framea,vec] (R) -- ++(0.75,-1.299) node[below right,inner sep=1pt]{$\mu$};

  \draw[frameb,vec] (R) -- ++(1.5,0)  node[right,inner sep=2pt]{$\bar\nu$};
  \draw[frameb,vec] (R) -- ++(0,-1.5) node[below,inner sep=2pt]{$N$};

\end{tikzpicture}

\end{figure}

On a boundary hypersurface $\mathcal S$ with outward unit normal $\eta$,
we use the convention
\[
  h^{\mathcal S}(X,Y)=g(\bar\nabla_X\eta,Y),
  \qquad H_{\mathcal S}=\operatorname{tr}_{g_{\mathcal S}}h^{\mathcal S}.
\]
All mean curvatures are unnormalized. In particular,
$H_\Sigma=\operatorname{tr}_{g_\Sigma}h^\Sigma$ and
$H_T=\operatorname{tr}_{g_T}h^T$. We also use $g$ for the induced metrics
and for the Euclidean metric when the ambient manifold is $\R^{n+1}$.
The classical Reilly formula takes the following form.

\begin{theorem}[\cite{Reilly77}]\label{thm:reilly-classical}
Let $\Omega\subset\R^{n+1}$ be a bounded domain with piecewise smooth boundary
and let $f\in C^1(\overline\Omega)\cap
C^\infty(\overline\Omega\setminus\Gamma)$ satisfy
$\bar\nabla^2f\in L^2(\Omega)$, and assume that the boundary terms below are integrable.
On each smooth hypersurface in $\partial\Omega$, let
$\eta$ be the outward unit normal. Then
\begin{equation*}
\int_\Omega\bigl[(\bar\Delta f)^2-|\bar\nabla^2f|^2\bigr]\,d\Omega
=
\sum_{\mathcal S\in\{\Sigma,T\}}\int_{\mathcal S}
\Bigl[
 f_\eta\,\Delta_{\mathcal S} f
 -g(\nabla^{\mathcal S} f_\eta,\nabla^{\mathcal S} f)
 +H_{\mathcal S}f_\eta^2
 +h^{\mathcal S}(\nabla^{\mathcal S} f,\nabla^{\mathcal S} f)
\Bigr]\,dA_{\mathcal S}.
\end{equation*}
Here $\partial\Omega=\Sigma\cup T$, and the differential operators
in each boundary integral are intrinsic to the corresponding hypersurface.
\end{theorem}

For the congruence argument, set $z=f|_\Sigma$ and $u=f_\nu$. It will be useful to integrate by parts to obtain a corner term,
\begin{equation}\label{eq:reilly-corner-prelim}
\begin{aligned}
&\int_\Sigma
\Bigl[u\,\Delta_\Sigma z-g(\nabla^\Sigma u,\nabla^\Sigma z)
+H_\Sigma u^2+h^\Sigma(\nabla^\Sigma z,\nabla^\Sigma z)\Bigr]\,dA
\\
&\qquad=
\int_\Sigma
\Bigl[2u\,\Delta_\Sigma z+H_\Sigma u^2+h^\Sigma(\nabla^\Sigma z,\nabla^\Sigma z)\Bigr]\,dA
-\int_\Gamma u\,g(\nabla^\Sigma z,\mu)\,ds.
\end{aligned}
\end{equation}

For the Heintze--Karcher inequalities we use the weighted Reilly formula
of Li and Xia~\cite{LiXiaReillySubstatic}, which extends the formula of
Qiu and Xia~\cite{QiuXia15} on Riemannian manifolds.
Let $U\subset\overline M$ be an open neighbourhood of $\overline\Omega$
and let $V\in C^\infty(U)$ be positive. Define
\begin{equation*}
\begin{aligned}
  L_Vf&:=\bar\Delta f-\frac{\bar\Delta V}{V}f,\\
  \mathcal A_V[f]&:=\bar\nabla^2f-\frac{f}{V}\bar\nabla^2V,\\
  W_V[f]&:=\bar\nabla f-\frac{f}{V}\bar\nabla V,
\end{aligned}
\end{equation*}
so that $\operatorname{tr}_{g}\mathcal A_V[f]=L_Vf$.

\begin{theorem}[\cite{LiXiaReillySubstatic}]
\label{thm:reilly-weighted}
Let $\Omega\subset\overline M^{n+1}$ be a relatively compact domain with
$\partial\Omega=\Sigma\cup T$ as above, and let $V\in C^\infty(U)$ be
positive on a neighbourhood $U$ of $\overline\Omega$. Suppose that
$f\in C^1(\overline\Omega)\cap C^\infty(\overline\Omega\setminus\Gamma)$,
$\bar\nabla^2f\in L^2(\Omega)$, and that the boundary terms below are
integrable. On each $\mathcal S\in\{\Sigma,T\}$, with outward unit normal
$\eta$, set
\[
  z=f|_{\mathcal S},\qquad
  \varphi=f_\eta-\frac{V_\eta}{V}z,\qquad
  Z=\nabla^{\mathcal S}z-\frac{z}{V}\nabla^{\mathcal S}V.
\]
Then
\begin{align}
&\int_\Omega V\Bigl[(L_Vf)^2-|\mathcal A_V[f]|^2\Bigr]\,d\Omega
\notag\\
&\quad=
\int_\Omega\mathcal Q_V\bigl(W_V[f],W_V[f]\bigr)\,d\Omega
\notag\\
&\qquad
+\sum_{\mathcal S\in\{\Sigma,T\}}\int_{\mathcal S}
V\left[
\varphi\left(\Delta_{\mathcal S}z-\frac{\Delta_{\mathcal S}V}{V}z\right)
-g(\nabla^{\mathcal S}\varphi,Z)
+H_{\mathcal S}\varphi^2
\right]\,dA_{\mathcal S}
\notag\\
&\qquad
+\sum_{\mathcal S\in\{\Sigma,T\}}\int_{\mathcal S}
V\left(h^{\mathcal S}-\frac{V_\eta}{V}g_{\mathcal S}\right)(Z,Z)\,dA_{\mathcal S},
\label{eq:reilly-weighted-prelim}
\end{align}
where
\begin{equation*}
  \mathcal Q_V
  :=(\bar\Delta V)g-\bar\nabla^2V+V\,\Ric.
\end{equation*}
\end{theorem}

Following \cite[Definition~1.2]{LiXiaReillySubstatic}, the triple
$(U,g,V)$ is called \emph{substatic} if
\[
  (\bar\Delta V)g-\bar\nabla^2V+V\Ric\geq0\quad\text{on }U.
\]
It is called \emph{static} if this tensor vanishes identically.
Since $V>0$ on $U$, the substatic condition is equivalently
$\Ric-V^{-1}\bar\nabla^2V+V^{-1}(\bar\Delta V)g\geq0$.
For the supporting region $T\subset S$, oriented by $N$, set
\begin{equation}\label{eq:V-boundary-tensor-prelim}
  \mathcal B_V^T
  :=h^T-\frac{V_N}{V}g_T.
\end{equation}
We call $T$ \emph{$V$-convex} if $\mathcal B_V^T\geq0$.


\section{Congruence}\label{sec:congruence}

The second
fundamental forms of $\Gamma$ in $\Sigma$ and $T$, with respect to $\mu$
and $\bar\nu$, are denoted by $h_\Gamma^\Sigma$ and $h_\Gamma^T$.
Their mean curvatures are
\[
H_\Gamma^\Sigma=\operatorname{tr}_{g_\Gamma}h_\Gamma^\Sigma,
\qquad
H_\Gamma^T=\operatorname{tr}_{g_\Gamma}h_\Gamma^T.
\]
For the second immersion
$\psi'$, we write $H_\Gamma^{T\,\prime}$ for the corresponding boundary curvature.
For a Euclidean isometric immersion $\psi'$, we recall that
\[
\vec H'=\operatorname{tr}_g\alpha'=\Delta_\Sigma\psi',
\]
where $\alpha'(X,Y)=(\bar\nabla_X\psi'_*Y)^\perp$.
For the first immersion, this convention gives $\vec H=-H_\Sigma\nu$.

\begin{theoremA}\label{thm:CongruenceTotGeo}
Let $\psi\colon\Sigma^n\to\R^{n+1}$ be a compact embedded hypersurface in Euclidean
space, $\theta$-capillary in a hyperplane $S$ with $\theta\in(0,\frac{\pi}{2}]$
and $\Sigma\cap S=\Gamma$, with mean curvature $H_\Sigma\geq 0$. Let
$\psi'\colon\Sigma^n\to\R^{n+p}$ be an isometric immersion with mean curvature vector
$\vec H'$ satisfying $|\vec H'|\leq H_\Sigma$, such that $\psi'(\Gamma)\subset\hat S$
for some hyperplane $\hat S\subset\R^{n+p}$ and $\psi'$ is $\theta$-capillary in
$\hat S$ with the same angle $\theta$. Then $\psi'$ differs from $\psi$ by a
rigid motion.
\end{theoremA}

\begin{proof}
The first immersion $\psi$, together with the hyperplane $S$, bounds a domain
$\Omega\subset\R^{n+1}$ with $\partial\Omega={\Sigma}\cup T$, where we identify $\Sigma$ with $\psi(\Sigma)$
and $T\subset S$ is a region of the hyperplane. We use the notation of \eqref{eq:capillarity} for the outward vector fields.

Following \cite{Ros88}, the idea is to extend the coordinates of the isometric
immersion $\psi'\colon\Sigma^n\to\R^{n+p}$ harmonically, but now we must take the boundary
into account. Let $\psi'\colon\Sigma^n\to\R^{n+p}$ be an isometric immersion as in the
statement of the theorem, with $\psi'(\Gamma)\subset\hat S$, a hyperplane in
$\R^{n+p}$ with unit normal $\hat N$. For each $i=1,\ldots,n+p$, let $F_i$ be the
solution of the mixed boundary value problem
\begin{equation*}
  \begin{cases}
    \bar\Delta F_i = 0 &\text{in }\Omega,\\
    F_i = \psi'_i\circ\psi^{-1} &\text{on }{\Sigma},\\
    \partial_N F_i = \hat N_i &\text{on }T.
  \end{cases}
\end{equation*}
The natural idea is to show that $F=(F_1,\ldots,F_{n+p})$ is a rigid motion, the
Neumann condition being precisely what forces this rigid motion to send the
normal of the hyperplane $S$ to the normal of the hyperplane $\hat S$.

Reilly's formula~\cite{Reilly77} gives
\begin{equation}\label{eq:Reilly-hyperplane}
  -\int_\Omega|\bar\nabla^2 F_i|^2\,d\Omega
  = \int_{\Sigma}\mathcal{I}_i^{\Sigma}\,dA + \int_T\mathcal{I}_i^T\,dA_T.
\end{equation}
On ${\Sigma}$, with restriction $z_i:=F_i|_{\Sigma}$, normal derivative
$u_i:=\partial_\nu F_i$, mean curvature $H_\Sigma$ and second fundamental form
$h^\Sigma$, the integrand is
\begin{equation}\label{eq:Reilly-Sigma}
  \mathcal{I}_i^{\Sigma} = u_i\,\Delta_\Sigma z_i
  - g(\nabla^\Sigma u_i,\,\nabla^\Sigma z_i)
  + H_\Sigma\,u_i^2
  + h^\Sigma(\nabla^\Sigma z_i,\,\nabla^\Sigma z_i).
\end{equation}
On the supporting region $T$, with restriction $w_i:=F_i|_T$, normal derivative
$v_i:=\partial_N F_i$, mean curvature $H_T$ and second fundamental form $h^T$,
the integrand is
\begin{equation}\label{eq:Reilly-T}
  \mathcal{I}_i^T = v_i\,\Delta_T w_i
  - g(\nabla^T v_i,\,\nabla^T w_i)
  + H_T\,v_i^2
  + h^T(\nabla^T w_i,\,\nabla^T w_i).
\end{equation}

We first study the terms on the flat region $T$. On $T\subset S$ we have
$h^T\equiv 0$ (totally geodesic) and $H_T=0$, and the normal derivative is
$v_i=\partial_N F_i=\hat N_i$. Since $\hat N_i$ is constant, the gradient
satisfies $\nabla^T v_i=\nabla^T\hat N_i=0$. The restriction $w_i=F_i|_T$ is
not prescribed but is determined by the solution, and a priori we do not know
that $F|_T$ is an isometric immersion. The four terms
in~\eqref{eq:Reilly-T} thus reduce to
\begin{equation*}
  \mathcal{I}_i^T = \hat N_i\,\Delta_T w_i.
\end{equation*}
By Green's formula on $T$,
\begin{equation*}
  \int_T\hat N_i\,\Delta_Tw_i\,dA_T
  = \hat N_i\int_T\Delta_Tw_i\,dA_T
  = \hat N_i\int_\Gamma\partial_{\bar\nu}w_i\,ds
  = \hat N_i\int_\Gamma\partial_{\bar\nu}F_i\,ds.
\end{equation*}
Summing over $i$,
\begin{equation}\label{eq:T-sum}
  \sum_i\int_T\mathcal{I}_i^T\,dA_T
  = \int_\Gamma g(\hat N,\,\partial_{\bar\nu}F)\,ds
  = \int_\Gamma g(\hat N,\,F_*\bar\nu)\,ds.
\end{equation}

On the region ${\Sigma}$ we have $z_i=F_i|_{\Sigma}=\psi'_i\circ\psi^{-1}$ and
$u_i=\partial_\nu F_i$. We write $\mathbf u:=(u_1,\ldots,u_{n+p})=\partial_\nu F|_{\Sigma}\in\R^{n+p}$.
Integrating by parts on ${\Sigma}$ as in \eqref{eq:reilly-corner-prelim} gives
\begin{equation}\label{eq:Sigma-after-IPP}
  \int_{\Sigma}\mathcal{I}_i^{\Sigma}\,dA
  = \int_{\Sigma}\bigl[2u_i\Delta_\Sigma z_i + H_\Sigma u_i^2
  + h^\Sigma(\nabla^\Sigma z_i,\nabla^\Sigma z_i)\bigr]\,dA
  - \int_\Gamma u_i\,\partial_\mu z_i\,ds.
\end{equation}
Combining~\eqref{eq:Reilly-hyperplane},~\eqref{eq:T-sum} and~\eqref{eq:Sigma-after-IPP},
\begin{equation}\label{eq:full-hyperplane}
  -\sum_i\int_\Omega|\bar\nabla^2 F_i|^2
  = \int_{\Sigma}\sum_i\bigl[2u_i\Delta_\Sigma z_i + H_\Sigma u_i^2
  + h^\Sigma(\nabla^\Sigma z_i,\nabla^\Sigma z_i)\bigr]\,dA + \int_\Gamma\mathcal{R}\,ds,
\end{equation}
where the corner residue is
\begin{equation*}
  \mathcal{R} = -g(\mathbf u,\,\psi'_*\mu) + g(\hat N,\,\partial_{\bar\nu}F)
  = -g(\mathbf u,\,\psi'_*\mu) + g(\hat N,\,F_*{\bar\nu}).
\end{equation*}
Here $\mathbf u=(u_1,\ldots,u_{n+p})=F_*\nu$ and we used that
$\psi'_*\mu=(\partial_\mu z_1,\ldots,\partial_\mu z_{n+p})$.

It remains to show that $\mathcal{R}$ is identically zero. The
projections of $\bar\nabla F_i$ onto the two bases $\{\nu,\mu\}$ and $\{N,\bar\nu\}$
of the normal plane to $\Gamma$ are
\begin{equation}\label{eq:proj-nu-mu}
  g(\bar\nabla F_i,\,\nu) = u_i, \qquad
  g(\bar\nabla F_i,\,\mu) = \partial_\mu z_i = (\psi'_*\mu)_i,
\end{equation}
\begin{equation}\label{eq:proj-N-nubar}
  g(\bar\nabla F_i,\,N) = \hat N_i, \qquad
  g(\bar\nabla F_i,\,\bar\nu) = \partial_{\bar\nu}F_i.
\end{equation}
Using the change of basis~\eqref{eq:capillarity} we obtain
\begin{align}
  u_i &= -\cos\theta\,\hat N_i + \sin\theta\,\partial_{\bar\nu}F_i,
    \label{eq:u-decomp}\\
  (\psi'_*\mu)_i &= \sin\theta\,\hat N_i + \cos\theta\,\partial_{\bar\nu}F_i.
    \label{eq:dpsi-decomp}
\end{align}
When $\theta=\frac{\pi}{2}$ we have $\mu=N$ and $\nu=\bar\nu$, and comparing the
projections~\eqref{eq:proj-nu-mu} and~\eqref{eq:proj-N-nubar} of $\bar\nabla F_i$
yields, for every $i$,
\[
  (\psi'_*\mu)_i = g(\bar\nabla F_i,\,\mu) = g(\bar\nabla F_i,\,N) = \hat N_i,
  \qquad
  u_i = g(\bar\nabla F_i,\,\nu) = g(\bar\nabla F_i,\,\bar\nu) = \partial_{\bar\nu}F_i,
\]
that is, $\psi'_*\mu=\hat N$ and $\mathbf u=\partial_{\bar\nu}F$. Substituting into the
residue,
\[
  \mathcal{R} = -g(\mathbf u,\,\psi'_*\mu) + g(\hat N,\,\partial_{\bar\nu}F)
  = -g(\mathbf u,\,\hat N) + g(\hat N,\,\mathbf u) = 0.
\]
We may now suppose $\theta\in(0,\frac{\pi}{2})$. From~\eqref{eq:dpsi-decomp} we
solve for $\partial_{\bar\nu}F_i$,
\begin{equation}\label{eq:nubar-solve}
  \partial_{\bar\nu}F_i = \frac{(\psi'_*\mu)_i - \sin\theta\,\hat N_i}{\cos\theta},
\end{equation}
and substituting into~\eqref{eq:u-decomp},
\begin{equation}\label{eq:U-solve}
  u_i = \frac{\sin\theta\,(\psi'_*\mu)_i - \hat N_i}{\cos\theta}.
\end{equation}
Using~\eqref{eq:U-solve} and~\eqref{eq:nubar-solve} in the two terms of
$\mathcal{R}$,
\begin{align*}
  g(\mathbf u,\,\psi'_*\mu) &= \sum_i u_i\,(\psi'_*\mu)_i
    = \frac{\sin\theta\,|\psi'_*\mu|^2 - g(\hat N,\,\psi'_*\mu)}{\cos\theta},
    \\[3pt]
  g(\hat N,\,\partial_{\bar\nu}F) &= \sum_i\hat N_i\,\partial_{\bar\nu}F_i
    = \frac{g(\hat N,\,\psi'_*\mu) - \sin\theta\,|\hat N|^2}{\cos\theta}.
\end{align*}
Therefore
\begin{align*}
  \cos\theta\cdot\mathcal{R}
    &= -\bigl(\sin\theta\,|\psi'_*\mu|^2 - g(\hat N,\,\psi'_*\mu)\bigr)
    + \bigl(g(\hat N,\,\psi'_*\mu) - \sin\theta\,|\hat N|^2\bigr)\\
    &= -\sin\theta\bigl(|\psi'_*\mu|^2+|\hat N|^2\bigr)
    + 2g(\hat N,\,\psi'_*\mu). 
\end{align*}
We now use the following three facts.
\begin{enumerate}
  \item $|\psi'_*\mu|=1$, since $\psi'$ is an isometric immersion and $\mu$ is a
        unit vector.
  \item $|\hat N|=1$, the unit normal to $\hat S$.
  \item $g(\hat N,\,\psi'_*\mu)=\sin\theta$, since $\psi'$ is $\theta$-capillary
        in $\hat S$, so the component of $\psi'_*\mu$ normal to $\hat S$ is
        $\sin\theta$.
\end{enumerate}
Hence
\begin{equation*}
  \cos\theta\cdot\mathcal{R} = -\sin\theta(1+1)+2\sin\theta
  = -2\sin\theta+2\sin\theta = 0,
\end{equation*}
so that $\mathcal{R}=0$.

The remaining terms on ${\Sigma}$ are handled as in Theorem~2 of
\cite{Ros88}. Summing the three interior terms on ${\Sigma}$
in~\eqref{eq:full-hyperplane}, we have
\[
  \sum_i u_i\Delta_\Sigma z_i = g(\mathbf u,\,\Delta_\Sigma\psi') = g(\mathbf u,\,\vec H'),
\]
where $\vec H'$ is the mean curvature vector of $\psi'$,
\[
  \sum_i H_\Sigma u_i^2 = H_\Sigma|\mathbf u|^2,
\]
and, for an orthonormal frame $\{e_\alpha\}$ on $\Sigma$, the isometry of $\psi'$ gives
\[
  \sum_i h^\Sigma(\nabla^\Sigma z_i,\nabla^\Sigma z_i)
  = \sum_{\alpha,\beta}h^\Sigma(e_\alpha,e_\beta)
  \underbrace{\sum_i (\psi'_*e_\alpha)_i\,(\psi'_*e_\beta)_i}
  _{=\,g(\psi'_*e_\alpha,\,\psi'_*e_\beta)\,=\,\delta_{\alpha\beta}}
  = \sum_\alpha h^\Sigma(e_\alpha,e_\alpha) = H_\Sigma.
\]
Completing the square, and using $g(\mathbf u,\,\vec H')\geq -|\mathbf u|\,|\vec H'|\geq -|\mathbf u|\,H_\Sigma$
from the hypothesis $|\vec H'|\leq H_\Sigma$,
\begin{align*}
  2g(\mathbf u,\,\vec H')+H_\Sigma|\mathbf u|^2+H_\Sigma
    &\geq -2H_\Sigma|\mathbf u|+H_\Sigma|\mathbf u|^2+H_\Sigma\\
    &= H_\Sigma\bigl(|\mathbf u|^2-2|\mathbf u|+1\bigr) = H_\Sigma(|\mathbf u|-1)^2\geq 0.
\end{align*}
From~\eqref{eq:full-hyperplane} with $\mathcal{R}=0$,
\begin{equation*}
  -\sum_i\int_\Omega|\bar\nabla^2 F_i|^2
  \;\geq\; \int_{\Sigma} H_\Sigma(|\mathbf u|-1)^2\,dA.
\end{equation*}
Since $H_\Sigma$ is nonnegative, both sides therefore vanish, giving the following.
\begin{itemize}
  \item $\bar\nabla^2 F_i=0$ in $\Omega$ for every $i$, so $F(x)=Bx+b$ with
        $B\colon\R^{n+1}\to\R^{n+p}$ linear and $b\in\R^{n+p}$.
  \item $|\mathbf u|=1$ at every point where $H_\Sigma>0$. At each such point, equality in the
        Cauchy-Schwarz estimate forces $\mathbf u$ to be parallel to $\vec H'$, hence
        normal to $\psi'(\Sigma)$.
\end{itemize}

It remains to check that $B$ is a linear isometry. Fix a point $p\in{\Sigma}$ with
$H_\Sigma>0$. Such a point exists, since otherwise the height function relative to $S$
would be harmonic with zero boundary data, forcing $\psi(\Sigma)\subset S$
and contradicting the positive contact angle. On ${\Sigma}$ we have $F|_{\Sigma}=\psi'\circ\psi^{-1}$, an
isometric immersion, so $B=F_*$ carries $T_p{\Sigma}$ isometrically onto its image.
Writing $\mathbf u=F_*\nu=B\nu$ for the outward unit normal $\nu$ of ${\Sigma}$, the relation
$|\mathbf u|=1$ says that $B$ preserves the norm of $\nu$, while the fact that $\mathbf u$ is
parallel to $\vec H'$ says that $B\nu$ is normal to the image of $\psi'_*$,
hence orthogonal to $B(T_p{\Sigma})$. Thus $B$ sends the orthonormal basis
$\{e_1,\ldots,e_n,\nu\}$ of $\R^{n+1}$, where $\{e_\alpha\}$ is an orthonormal
basis of $T_p{\Sigma}$, to an orthonormal set in $\R^{n+p}$. Therefore $B$ is a
linear isometry. Consequently $\psi'=B\circ\psi+b$, and the two immersions are congruent.
\end{proof}

For a hyperplane support, the proof itself shows that the extension is
isometric on $T$. For a spherical support, we need first to obtain an isometric
extension on each component of $T$ using Theorem~A in
\cite{FreitasGuimaraesCongruence}. The following lemma provides the required
comparison of boundary mean curvatures that allows us to do this.

\begin{lemma}\label{lem:capillary-boundary}
  Let \(\psi,\psi'\colon \Sigma^n\to\R^{n+1}\) be
  isometric immersions of a compact Riemannian manifold with
  boundary, both \(\theta\)-capillary in the unit sphere
  \(\Sph^n\) with \(\theta\in(0,\pi/2]\).  Denote by
  \(H_\Gamma^T\) and
  \(H_\Gamma^{T\,\prime}\) the mean curvatures
  of \(\Gamma^{n-1}\) as a hypersurface of \(\Sph^n\)
  under \(\psi\) and \(\psi'\), respectively.  Then
  \begin{enumerate}
    \item[(i)]\label{it:notFree}
          if \(\theta\neq\pi/2\), then
          \(H_\Gamma^T
           =H_\Gamma^{T\,\prime}\).
    \item[(ii)]\label{it:free}
          if \(\theta=\pi/2\) and both immersions are minimal, then
          \(\bigl(H_\Gamma^T\bigr)^{2}
           =\bigl(H_\Gamma^{T\,\prime}\bigr)^{2}\).
  \end{enumerate}
\end{lemma}

\begin{proof}
  We express the boundary mean curvature in terms of the intrinsic
  geometry of $\Sigma$ and the contact angle.

  For \(X,Y\in T(\Gamma)\), differentiating
  the capillarity relations~\eqref{eq:capillarity} and using
  that \(h^T(X,Y)
  =g(X,Y)\) yields
  \begin{equation}\label{eq:hMR}
    h^\Sigma(X,Y)
      = \sin\theta\;h_\Gamma^T(X,Y)
        - \cos\theta\,g(X,Y),
  \end{equation}
  \begin{equation}\label{eq:hGammaM}
    h_\Gamma^\Sigma(X,Y)
      = \cos\theta\;h_\Gamma^T(X,Y)
        + \sin\theta\,g(X,Y).
  \end{equation}
  Taking the trace of~\eqref{eq:hGammaM} over
  \(T(\Gamma)\) gives
  \begin{equation*}
    H_\Gamma^\Sigma
      = \cos\theta\;H_\Gamma^T
        + (n-1)\sin\theta.
  \end{equation*}
  If \(\theta\neq\pi/2\), then
  \[
    H_\Gamma^T
      = \frac{H_\Gamma^\Sigma-(n-1)\sin\theta}
             {\cos\theta}\,,
  \]
  where the right hand side depends only on the intrinsic
  geometry of \(\Sigma^n\) and on~\(\theta\).  In particular it is
  the same for \(\psi\) and~\(\psi'\), so
  \(H_\Gamma^T
   =H_\Gamma^{T\,\prime}\).

  It remains to consider the free boundary case
  \(\theta=\pi/2\).  We first observe that
  \(h^\Sigma(X,\mu)=0\) for every
  \(X\in T(\Gamma)\).  Indeed, since
  \(g(\nu,N)=-\cos\theta\) is constant along
  \(\Gamma^{n-1}\), differentiating in the direction
  of~\(X\) gives
  \[
    g( \bar\nabla_X\nu,N)
    + g(\nu,\bar\nabla_X N) = 0.
  \]
  Because \(N=\psi\) is the position vector on
  \(\Sph^n\) we have
  \(\bar\nabla_X N=X\), and
  \(g(\nu,X)=0\) since \(\nu\) is the normal vector field of~\(\Sigma^n\).  Hence
  \(g(\bar\nabla_X\nu,N)=0\).
  Inverting the capillarity relations gives
  \(N=\sin\theta\,\mu-\cos\theta\,\nu\), so
  \[
    \sin\theta\,
    g(\bar\nabla_X\nu,\mu)
    -\cos\theta\,
    \underbrace{
      g(\bar\nabla_X\nu,\nu)
    }_{=\,0}
    = 0.
  \]
  As \(\sin\theta\neq 0\) we conclude
  \(g(\bar\nabla_X\nu,\mu)=0\),
  which means
  \(h^\Sigma(X,\mu)=0\).

  Along \(\Gamma^{n-1}\) consider the orthonormal frame
  \(\{e_1,\dots,e_{n-1},\mu\}\) of \(T\Sigma\), where
  \(e_i\in T(\Gamma)\).  The mean curvature
  of~\(\Sigma^n\) decomposes as
  \[
    H_\Sigma
      = \sum_{i=1}^{n-1}
        h^\Sigma(e_i,e_i)
        + h^\Sigma(\mu,\mu).
  \]
  Setting \(\theta=\pi/2\) in~\eqref{eq:hMR} gives
  \(h^\Sigma(e_i,e_i)
   =h_\Gamma^T(e_i,e_i)\),
  so
  \(\sum_{i}h^\Sigma(e_i,e_i)
   =H_\Gamma^T\).
  Since $h^\Sigma(e_i,\mu)=0$, the Gauss equation gives the identity
  \[
    \Ric^\Sigma(\mu,\mu)
    =h^\Sigma(\mu,\mu)H_\Gamma^T
    =(H_\Sigma-H_\Gamma^T)H_\Gamma^T.
  \]
  If the immersion is minimal, this becomes
  \[
    \Ric^\Sigma(\mu,\mu)=-\bigl(H_\Gamma^T\bigr)^2.
  \]
  The left hand side is intrinsic, which proves item~\textit{(ii)}.
\end{proof}

We are now in a position to prove Theorem~\ref{thm:free-boundary-sphere}.

\begin{theoremB}\label{thm:free-boundary-sphere}
Let $\psi\colon \Sigma^n\to\R^{n+1}$ be a compact embedded minimal hypersurface
with $\psi(\Sigma)\subset\overline{\mathbb{B}}^{\,n+1}$ and free boundary in the
unit sphere $\Sph^n$, and set $\Gamma:=\partial\Sigma$. Assume that each
connected component of $\Gamma$ bounds a mean convex domain of $\Sph^n$
contained in an open hemisphere, and that these domains are pairwise disjoint.
We denote their union by $T\subset\Sph^n$, so that $\partial T=\Gamma$. If
$\psi'\colon \Sigma^n\to\R^{n+1}$ is a minimal isometric immersion with free
boundary in $\Sph^n$, then $\psi'$ differs from $\psi$ by a rigid motion.
\end{theoremB}

\begin{proof}
Let $\Omega$ be the domain bounded by $\Sigma\cup T$, with the
orientations of Section~\ref{sec:preliminaries}. On $T$ the outward unit
normal is $N=x$. The free boundary condition gives
\begin{equation}\label{eq:fb}
  \mu=N,\qquad \nu=\bar\nu\qquad\text{along }\Gamma.
\end{equation}
By Lemma~\ref{lem:capillary-boundary}(ii),
$\bigl(H_\Gamma^T\bigr)^2=\bigl(H_\Gamma^{T\,\prime}\bigr)^2$.
Since $H_\Gamma^T\geq0$, we have
$H_\Gamma^T=|H_\Gamma^{T\,\prime}|$.
Applying Theorem~A in \cite{FreitasGuimaraesCongruence} to each connected
component of $\Gamma$ gives a map $G\colon T\to\Sph^n$ whose restriction
to each connected component of $T$ is the restriction of an ambient
isometry of $\Sph^n$, and
\[
G|_\Gamma=\psi'\circ\psi^{-1}|_\Gamma.
\]

We extend the coordinates of $\psi'$ harmonically, prescribing $G$ on
$T$ to control the second fundamental form term in Reilly's formula.
For $i=1,\ldots,n+1$, let $F_i$ be the solution of
\[
  \begin{cases}
    \bar\Delta F_i=0 & \text{in }\Omega,\\
    F_i=\psi'_i\circ\psi^{-1} & \text{on }\Sigma,\\
    F_i=G_i & \text{on }T.
  \end{cases}
\]
The data agree on $\Gamma$. Existence, uniqueness and the regularity
needed below follow from Remark~\ref{rem:congruence-dirichlet-regularity}.
Since $F|_\Sigma=\psi'\circ\psi^{-1}$, it remains to prove that $F$ is a
rigid motion.

Reilly's formula gives
\begin{equation}\label{eq:Reilly-free-sphere}
  -\int_\Omega|\bar\nabla^2 F_i|^2\,d\Omega
  = \int_{\Sigma}\mathcal{I}_i^\Sigma\,dA + \int_T\mathcal{I}_i^T\,dA_T,
\end{equation}
where
\[ \mathcal{I}^{\Sigma}_{i} = u_i\Delta_\Sigma z_i-g(\nabla^\Sigma u_i,\nabla^\Sigma z_i)+H_\Sigma u_i^2
  +h^\Sigma(\nabla^\Sigma z_i,\nabla^\Sigma z_i),\]
\[ \mathcal{I}^{T}_{i} = v_i\Delta_T w_i-g(\nabla^T v_i,\nabla^T w_i)+H_T v_i^2
  +h^T(\nabla^T w_i,\nabla^T w_i),\]
with $z_i = F_i|_{\Sigma} = \psi'_i\circ\psi^{-1}$, normal derivative $u_i = \partial_\nu F_i$,
mean curvature $H_\Sigma$ and second fundamental form $h^\Sigma$ of $\Sigma^n$ in $\R^{n+1}$,
and with $w_i = F_i|_{T} = G_i$, normal derivative $v_i = \partial_N F_i$, mean
curvature $H_T$ and second fundamental form $h^T$ of $T$ in $\R^{n+1}$. Applying \eqref{eq:reilly-corner-prelim} on $\Sigma$ and its analogue
on $T$ gives the interior terms with $2u_i\Delta_\Sigma z_i$ and
$2v_i\Delta_Tw_i$, together with the corresponding integrals over $\Gamma$.

Summing over $i$, the interior terms are computed exactly as in
Theorem~\ref{thm:CongruenceTotGeo}. On ${\Sigma}$ we have
$z_i=\psi'_i\circ\psi^{-1}$ and $u_i=\partial_\nu F_i$, and we write
$\mathbf u:=F_*\nu$. Since $\psi'$ is a minimal immersion,
$\Delta_\Sigma\psi'=0$, and since ${\Sigma}$ is minimal, $H_\Sigma=0$. Hence
\[
\begin{gathered}
  \sum_i 2u_i\Delta_\Sigma z_i = 2g(\mathbf u,\,\Delta_\Sigma\psi') = 0,
  \\
  \sum_i H_\Sigma u_i^2 = H_\Sigma|\mathbf u|^2 = 0,
  \\
  \sum_i h^\Sigma(\nabla^\Sigma z_i,\nabla^\Sigma z_i) = \operatorname{tr}h^\Sigma = H_\Sigma = 0,
\end{gathered}
\]
the first vanishing by minimality of $\psi'$ and the other two by
minimality of ${\Sigma}$. The integral over ${\Sigma}$ therefore contributes only its
corner term,
\[
  \sum_i\int_{\Sigma}\mathcal{I}_i^\Sigma
  = -\int_\Gamma g(\mathbf u,\,\psi'_*\mu)\,ds.
\]
On $T$ we have $w_i=G_i$ and $v_i=\partial_N F_i$, and we write $\mathbf v:=F_*N$.
Since $T\subset\Sph^n$ is totally umbilical, $h^T=g_T$ and
$H_T=n$, and since $G$ is an isometry into the unit sphere, $\Delta_T G=-nG$
and $|G|=1$. Summing over $i$,
\[
\begin{gathered}
  \sum_i 2v_i\Delta_T w_i = 2g(\mathbf v,\,\Delta_T G) = -2ng(\mathbf v,\,G),
  \\
  \sum_i H_T v_i^2 = n|\mathbf v|^2,
  \\
  \sum_i h^T(\nabla^T w_i,\nabla^T w_i) = \operatorname{tr}h^T = n,
\end{gathered}
\]
so the three interior terms combine into a square,
\[
  -2ng(\mathbf v,\,G)+n|\mathbf v|^2+n = n\bigl(|\mathbf v|^2-2g(\mathbf v,\,G)+|G|^2\bigr) = n|\mathbf v-G|^2,
\]
and $T$ contributes
$\sum_i\int_T\mathcal{I}_i^T = \int_T n|\mathbf v-G|^2\,dA_T
-\int_\Gamma g(\mathbf v,\,G_*\bar\nu)\,ds$.

Adding both contributions, \eqref{eq:Reilly-free-sphere} becomes
\begin{equation}\label{eq:full-free-sphere}
  -\sum_i\int_\Omega|\bar\nabla^2 F_i|^2
  = \int_T n|\mathbf v-G|^2\,dA_T + \int_\Gamma\mathcal{R}\,ds,
  \qquad
  \mathcal{R} = -g(\mathbf u,\,\psi'_*\mu) - g(\mathbf v,\,G_*\bar\nu).
\end{equation}

The corner residue vanishes. Indeed, the relations~\eqref{eq:fb} give
$\mathbf v=F_*N=F_*\mu=\psi'_*\mu$ and $\mathbf u=F_*\nu=F_*\bar\nu=G_*\bar\nu$, so
$\mathcal{R}=-2g(\mathbf u,\,\mathbf v)$. Now $\mathbf u=G_*\bar\nu$ is tangent to $\Sph^n$, while
$\mathbf v=\psi'_*\mu$ is the outward conormal of $\psi'(\Gamma)$ in
$\psi'(\Sigma)$, which by the free boundary condition is normal to $\Sph^n$.
Hence $g(\mathbf u,\,\mathbf v)=0$ and $\mathcal{R}=0$.

Therefore the right hand side of~\eqref{eq:full-free-sphere} equals
$\int_T n|\mathbf v-G|^2\geq 0$, while its left hand side is $\leq 0$. Hence both
vanish, so $\bar\nabla^2 F_i=0$ for every $i$ and $\mathbf v=G$ on $T$. From
$\bar\nabla^2 F_i=0$ we get $F(x)=Bx+b$ for a linear map
$B\colon\R^{n+1}\to\R^{n+1}$ and $b\in\R^{n+1}$. On $T$ the outward unit normal is
the position vector, $N=x$, so $\mathbf v=F_*N=Bx$, while $G=F|_T=Bx+b$. The identity
$\mathbf v=G$ then forces $b=0$, hence $F=B$ is linear.

It remains to check that $B$ is orthogonal. At a point $p$ in the interior
of $T$, the restriction of $B$ to $T_pT$ is an isometry because $B|_T=G$.
Moreover, $BN=G(p)$ is a unit vector orthogonal to $G_*(T_pT)$. Thus $B$
sends an orthonormal basis of $T_pT\oplus\operatorname{span}\{N\}$ to an
orthonormal basis of $\R^{n+1}$. Consequently $\psi'=B\circ\psi$, and the two
immersions are congruent.
\end{proof}

The same argument gives the following congruence theorem under a mean
curvature comparison, without assuming minimality. Here the contact
angle satisfies $0<\theta<\pi/2$, so Lemma~\ref{lem:capillary-boundary}(i) supplies the
boundary curvature comparison needed to construct $G$. The proofs of the two theorems share a substantial part. We present them separately so that the case distinction does not interrupt the arguments.

\begin{theoremC}\label{thm:capillary-sphere}
Let $\psi\colon \Sigma^n\to\R^{n+1}$ be a compact embedded hypersurface with
$\psi(\Sigma)\subset\overline{\mathbb{B}}^{\,n+1}$ and mean curvature $H_\Sigma\geq 0$,
$\theta$-capillary in the unit sphere $\Sph^n$ with
$\theta\in(0,\tfrac{\pi}{2})$, and set $\Gamma:=\partial\Sigma$. Assume that
each connected component of $\Gamma$ bounds a mean convex domain of $\Sph^n$
contained in an open hemisphere, and that these domains are pairwise disjoint.
We denote their union by $T\subset\Sph^n$, so that $\partial T=\Gamma$. If
$\psi'\colon \Sigma^n\to\R^{n+1}$ is an isometric immersion with mean curvature
vector $\vec H'$ satisfying $|\vec H'|\leq H_\Sigma$, $\theta$-capillary in $\Sph^n$
with the same angle $\theta$, then $\psi'$ differs from $\psi$ by a rigid
motion.
\end{theoremC}

\begin{proof}
Let $\Omega$ be the domain bounded by $\Sigma\cup T$, with the
orientations of Section~\ref{sec:preliminaries}. Thus $N=x$ on $T$, and
we use the capillary relations~\eqref{eq:capillarity}.
By Lemma~\ref{lem:capillary-boundary}(i),
$H_\Gamma^T=H_\Gamma^{T\,\prime}$. Since $H_\Gamma^T\geq0$, Theorem~A in
\cite{FreitasGuimaraesCongruence}, applied to each connected component
of $\Gamma$, gives a map $G\colon T\to\Sph^n$. On each connected component
of $T$, this map is the restriction of an ambient isometry of $\Sph^n$,
and
\[
G|_\Gamma=\psi'\circ\psi^{-1}|_\Gamma.
\]

As in Theorem~\ref{thm:free-boundary-sphere}, for $i=1,\ldots,n+1$ let
$F_i$ be the solution of
\[
  \begin{cases}
    \bar\Delta F_i=0 & \text{in }\Omega,\\
    F_i=\psi'_i\circ\psi^{-1} & \text{on }\Sigma,\\
    F_i=G_i & \text{on }T.
  \end{cases}
\]
The data agree on $\Gamma$, so
Remark~\ref{rem:congruence-dirichlet-regularity} gives existence,
uniqueness and the regularity needed below. We again prove that
$F=(F_1,\ldots,F_{n+1})$ is a rigid motion.

Reilly's formula~\cite{Reilly77} gives
\begin{equation}\label{eq:Reilly-capillary-sphere}
  -\int_\Omega|\bar\nabla^2 F_i|^2\,d\Omega
  = \int_{\Sigma}\mathcal{I}_i^{\Sigma}\,dA + \int_T\mathcal{I}_i^T\,dA_T,
\end{equation}
where $\mathcal I_i^\Sigma$ and $\mathcal I_i^T$ are given by
\eqref{eq:Reilly-Sigma} and \eqref{eq:Reilly-T}, with
\[
z_i=F_i|_\Sigma,\quad u_i=\partial_\nu F_i,\qquad
w_i=F_i|_T,\quad v_i=\partial_N F_i.
\]
We apply \eqref{eq:reilly-corner-prelim} on $\Sigma$ and its analogue on
$T$, as in the preceding proof.

On ${\Sigma}$ put $z_i=\psi'_i\circ\psi^{-1}$, $u_i=\partial_\nu F_i$ and
$\mathbf u:=F_*\nu$. Using $\Delta_\Sigma\psi'=\vec H'$ and the isometry of $\psi'$, we obtain
\[
  \sum_i 2u_i\Delta_\Sigma z_i=2g(\mathbf u,\,\vec H'),\qquad
  \sum_i H_\Sigma u_i^2=H_\Sigma|\mathbf u|^2,\qquad
  \sum_i h^\Sigma(\nabla^\Sigma z_i,\nabla^\Sigma z_i)=H_\Sigma,
\]
and, using $|\vec H'|\leq H_\Sigma$, we complete the square,
\[
\begin{aligned}
  2g(\mathbf u,\vec H')+H_\Sigma|\mathbf u|^2+H_\Sigma
  &\geq -2H_\Sigma|\mathbf u|+H_\Sigma|\mathbf u|^2+H_\Sigma\\
  &=H_\Sigma(|\mathbf u|-1)^2\geq0.
\end{aligned}
\]
Thus the ${\Sigma}$ component gives
\[
  \sum_i\int_{\Sigma}\mathcal{I}_i^{\Sigma}
  \geq \int_{\Sigma} H_\Sigma(|\mathbf u|-1)^2\,dA-\int_\Gamma g(\mathbf u,\,\psi'_*\mu)\,ds.
\]

On $T$ put $w_i=G_i$, $v_i=\partial_N F_i$ and $\mathbf v:=F_*N$. Since
$T\subset\Sph^n$ is totally umbilical we have $h^T=g_T$ and $H_T=n$, and
since $G$ is an isometry into the unit sphere we have $\Delta_T G=-nG$ and
$|G|=1$. Summing over $i$,
\[
  \sum_i 2v_i\Delta_T w_i=-2ng(\mathbf v,\,G),\qquad
  \sum_i H_T v_i^2=n|\mathbf v|^2,\qquad
  \sum_i h^T(\nabla^T w_i,\nabla^T w_i)=n,
\]
which combine into a square,
\[
  -2ng(\mathbf v,\,G)+n|\mathbf v|^2+n=n\bigl(|\mathbf v|^2-2g(\mathbf v,\,G)+|G|^2\bigr)=n|\mathbf v-G|^2,
\]
so the $T$ component gives
\[
  \sum_i\int_T\mathcal{I}_i^T=\int_T n|\mathbf v-G|^2\,dA_T
  -\int_\Gamma g(\mathbf v,\,G_*\bar\nu)\,ds.
\]

Adding the two boundary components, \eqref{eq:Reilly-capillary-sphere} yields
\begin{equation}\label{eq:full-capillary-sphere}
  -\sum_i\int_\Omega|\bar\nabla^2 F_i|^2
  \geq \int_{\Sigma} H_\Sigma(|\mathbf u|-1)^2\,dA + \int_T n|\mathbf v-G|^2\,dA_T
  + \int_\Gamma\mathcal{R}\,ds,
\end{equation}
with corner residue $\mathcal{R}=-g(\mathbf u,\,\psi'_*\mu)-g(\mathbf v,\,G_*\bar\nu)$.

The residue vanishes. The Dirichlet data give
$F_*\mu=\psi'_*\mu$ and $F_*\bar\nu=G_*\bar\nu$.
Since $F$ is $C^1$ up to $\Gamma$, inverting~\eqref{eq:capillarity} gives
$\nu=\tfrac{1}{\sin\theta}(\bar\nu-\cos\theta\,\mu)$ and
$N=\tfrac{1}{\sin\theta}(\mu-\cos\theta\,\bar\nu)$, hence
\[
\begin{aligned}
  \mathbf u=F_*\nu
  &=\tfrac{1}{\sin\theta}(G_*\bar\nu-\cos\theta\,\psi'_*\mu),\\
  \mathbf v=F_*N
  &=\tfrac{1}{\sin\theta}(\psi'_*\mu-\cos\theta\,G_*\bar\nu).
\end{aligned}
\]
Both $\psi'_*\mu$ and $G_*\bar\nu$ are unit vectors. To identify their
relative orientation, we use the capillary condition, which gives the unit
conormal $(\psi'_*\mu-\sin\theta\,G)/\cos\theta$ to $\psi'(\Gamma)$ in $\Sph^n$.
Thus
\[
\frac{\psi'_*\mu-\sin\theta\,G}{\cos\theta}
=\pm G_*\bar\nu
\]
on each connected component. The mean curvature with respect to the
conormal on the left is $H_\Gamma^{T\,\prime}=H_\Gamma^T$ by
Lemma~\ref{lem:capillary-boundary}(i). The mean curvature with respect
to $G_*\bar\nu$ is also $H_\Gamma^T$, since $G$ is an isometry on that
component. The negative sign would force $H_\Gamma^T\equiv0$, which is
impossible for a compact hypersurface contained in an open hemisphere.
Thus $\psi'_*\mu=\sin\theta\,G+\cos\theta\,G_*\bar\nu$ and
$g(\psi'_*\mu,G_*\bar\nu)=\cos\theta$. Therefore
\[
\begin{aligned}
  g(\mathbf u,\,\psi'_*\mu)
  &=\tfrac{1}{\sin\theta}\bigl(g(G_*\bar\nu,\,\psi'_*\mu)
  -\cos\theta\,|\psi'_*\mu|^2\bigr)=0,\\
  g(\mathbf v,\,G_*\bar\nu)
  &=\tfrac{1}{\sin\theta}\bigl(g(\psi'_*\mu,\,G_*\bar\nu)
  -\cos\theta\,|G_*\bar\nu|^2\bigr)=0,
\end{aligned}
\]
so $\mathcal{R}=0$.

Consequently the right hand side of~\eqref{eq:full-capillary-sphere} is nonnegative while the
left hand side is $\leq 0$. Both vanish, and so do the two squares. Thus
$\bar\nabla^2 F_i=0$ in $\Omega$, giving $F(x)=Bx+b$ with $B\colon\R^{n+1}\to\R^{n+1}$
linear and $b\in\R^{n+1}$, and $\mathbf v=G$ on $T$. From here the argument is the same
as in the proof of Theorem~\ref{thm:free-boundary-sphere} and, using these two facts alone,
gives $b=0$ and shows that $B$ is orthogonal. Consequently $\psi'=B\circ\psi$ and the two
immersions are congruent.
\end{proof}

\begin{remark}\label{rem:congruence-dirichlet-regularity}
For the Dirichlet problems in Theorems~\ref{thm:free-boundary-sphere}
and~\ref{thm:capillary-sphere}, the smooth boundary data agree on $\Gamma$.
Since $\Sigma$ and $T$ meet transversely, these data admit a common smooth
extension to $\overline\Omega$. Subtracting it gives homogeneous Dirichlet
data and a smooth source, so existence and uniqueness follow from
Poincar\'e's inequality and the Lax-Milgram theorem. Since $0<\theta\leq\pi/2$, the first transverse
exponent is $\pi/\theta\geq2$, and the standard corner estimates give
$F_i\in C^1(\overline\Omega)\cap W^{2,p}(\Omega)$ for every finite $p>1$
(see~\cite{Grisvard85,Dauge88}).

These solutions have the regularity described in the last paragraph of
Remark~\ref{rem:HK-mixed-regularity}, which justifies the integrations
and traces used above.
\end{remark}

\section{Capillary Heintze--Karcher inequalities in substatic manifolds}\label{sec:substatic-HK}

We use the notation of Section~\ref{sec:preliminaries} and assume that
$V$ is smooth and positive on a neighbourhood of $\overline\Omega$.
For $c\in\R$, consider
\begin{equation}\label{eq:mixed-problem}
\begin{cases}
L_V f=1 & \text{in }\Omega,\\[2mm]
f=0 & \text{on }\Sigma,\\[1mm]
f_N-\dfrac{V_N}{V}f=c & \text{on }T.
\end{cases}
\end{equation}
Writing $f=Vw$, we have
\begin{equation*}
L_V(Vw)=V^{-1}\operatorname{div}(V^2\bar\nabla w),
\qquad
(Vw)_N-\frac{V_N}{V}(Vw)=Vw_N.
\end{equation*}
Hence \eqref{eq:mixed-problem} is equivalent to
\begin{equation}\label{eq:weighted-mixed}
\begin{cases}
\operatorname{div}(V^2\bar\nabla w)=V & \text{in }\Omega,\\
w=0 & \text{on }\Sigma,\\
Vw_N=c & \text{on }T.
\end{cases}
\end{equation}

Problem~\eqref{eq:mixed-problem} has a unique solution with the regularity
needed for Reilly's formula, as explained in
Remark~\ref{rem:HK-mixed-regularity}. This is where we use the restriction
$0<\theta\leq\pi/2$, with $c=0$ when $\theta=\pi/2$.
The geometric approach in \cite{HuWeiXiaZhou2025} avoids this auxiliary
boundary value problem.
\begin{remark}\label{rem:HK-mixed-regularity}
The mixed boundary value problems used in this paper have unique solutions with the
regularity required by the Reilly identities. For \eqref{eq:weighted-mixed},
the weak formulation is
\[
\int_\Omega V^2g(\bar\nabla w,\bar\nabla\phi)\,d\Omega
=-\int_\Omega V\phi\,d\Omega+c\int_TV\phi\,dA_T
\]
for $\phi\in H^1(\Omega)$ vanishing on $\Sigma$.
Since $V$ is bounded away from zero, Poincar\'e's inequality and the
Lax-Milgram theorem give existence and uniqueness. For $0<\theta<\pi/2$,
we lift the boundary data and apply the classical mixed boundary
estimates of \cite{Lieberman86,Lieberman89}, following the argument in
\cite[Appendix~A]{JiaXiaZhang2023}. For $\theta=\pi/2$, we take $c=0$.
The homogeneous Neumann condition for $w$ allows local reflection in
Fermi coordinates, followed by Dirichlet boundary estimates.
The harmonic mixed problem in Theorem~\ref{thm:CongruenceTotGeo} is treated
in the same way after lifting its boundary data. When $\theta=\pi/2$,
the free boundary condition gives the compatibility at $\Gamma$ needed
to reflect across the supporting hyperplane.

In these mixed problems, the solutions are $C^1$ up to the corner, their
Hessians are square integrable, and the second derivatives occurring in
the boundary terms are integrable on each boundary hypersurface.
These estimates justify Reilly's formula, tangential integration by
parts and the first derivative traces at $\Gamma$. Full $C^2$ regularity
at the corner when $\theta=\pi/2$ is not needed.
\end{remark}

To simplify the calculations below, set
\begin{equation*}
A:=\int_\Omega V\,d\Omega,
\qquad
B:=\int_TV\,dA_T,
\qquad
C:=\int_\Gamma V\,ds,
\end{equation*}
and
\begin{equation*}
\mathcal K_\theta
:=\tan\theta\,C-\int_T H_TV\,dA_T.
\end{equation*}
The following is the main result of this section.

\begin{theoremD}\label{thm:substatic-HK}
Let $(\overline M^{n+1},g,V)$ be substatic on $\Omega$, with $V>0$.  Let
$\partial\Omega=\Sigma\cup T$ be as above, and assume the following.
\begin{enumerate}[label=\rm(\roman*)]
\item $\Sigma$ is compact, embedded and $H_\Sigma>0$.
\item $\Sigma$ meets $T$ at a constant angle $\theta\in(0,\pi/2)$.
\item $T$ is $V$-convex in the sense of \eqref{eq:V-boundary-tensor-prelim}.
\end{enumerate}
Then $\mathcal K_\theta>0$ and
\begin{equation}\label{eq:main-HK}
\int_\Sigma\frac{V}{H_\Sigma}\,dA
\ge
\frac{n+1}{n}\int_\Omega V\,d\Omega
+
\frac{\left(\displaystyle\int_TV\,dA_T\right)^2}
{\displaystyle
\tan\theta\int_\Gamma V\,ds-\int_TH_TV\,dA_T}.
\end{equation}
If equality holds, then $\Sigma$ is totally umbilical. 
\end{theoremD}

\begin{proof}
Fix $c\in\R$ and let $f$ be the solution of \eqref{eq:mixed-problem}.  Since
$\operatorname{tr}\mathcal A_V[f]=L_Vf=1$, the matrix Cauchy-Schwarz inequality gives
\begin{equation}\label{eq:trace-CS}
|\mathcal A_V[f]|^2\ge\frac1{n+1}.
\end{equation}
Therefore the left hand side of \eqref{eq:reilly-weighted-prelim} satisfies
\begin{equation*}
\int_\Omega V\big(1-|\mathcal A_V[f]|^2\big)\,d\Omega
\le \frac{n}{n+1}A.
\end{equation*}
Since $\mathcal Q_V\ge0$, the sum of the boundary terms is bounded above by the same quantity.

On $\Sigma$ we have $f=0$, hence $z=0$, $Z=0$ and $\varphi=f_\nu$.  The only surviving term is
\begin{equation*}
\int_\Sigma VH_\Sigma f_\nu^2\,dA.
\end{equation*}

On $T$ we have
\[
\varphi=f_N-\frac{V_N}{V}f=c,
\]
so $\nabla^T\varphi=0$. The $V$-convexity assumption ensures that the
last term in \eqref{eq:reilly-weighted-prelim} is nonnegative.
To evaluate the first boundary term, note that $f=0$ on $\Sigma$ gives
$\bar\nabla f=f_\nu\nu$ along $\Gamma$. Since $f=0$ on $\Gamma$,
the Robin condition and \eqref{eq:capillarity} give
\[
c=f_N=-\cos\theta\,f_\nu,\qquad
f_\nu=-\frac{c}{\cos\theta},\qquad
f_{\bar\nu}=\sin\theta\,f_\nu=-c\tan\theta.
\]
Green's formula on $T$ now yields
\begin{align*}
\int_TVc\left(\Delta_Tf-\frac{\Delta_TV}{V}f\right)dA_T
&=c\int_\Gamma\left(Vf_{\bar\nu}-fV_{\bar\nu}\right)ds \\
&=c\int_\Gamma V f_{\bar\nu}\,ds
=-c^2\tan\theta\,C,
\end{align*}
where we used $f=0$ on $\Gamma$. The mean-curvature term on $T$ is
\[
c^2\int_TH_TV\,dA_T.
\]
Consequently,
\begin{equation}\label{eq:key-Reilly-estimate}
\int_\Sigma VH_\Sigma f_\nu^2\,dA
\le
\frac{n}{n+1}A+c^2\mathcal K_\theta.
\end{equation}

Next, Green's identity for $V$ and $f$ gives
\begin{align*}
A
&=\int_\Omega V L_Vf\,d\Omega 
=\int_{\partial\Omega}(Vf_\eta-fV_\eta)\,dA \\
&=\int_\Sigma Vf_\nu\,dA
+\int_TV\left(f_N-\frac{V_N}{V}f\right)dA_T
=\int_\Sigma Vf_\nu\,dA+cB.
\end{align*}
Thus
\begin{equation*}
\int_\Sigma Vf_\nu\,dA=A-cB.
\end{equation*}
By Cauchy-Schwarz,
\begin{align}
(A-cB)^2
&\le
\left(\int_\Sigma\frac{V}{H_\Sigma}\,dA\right)
\left(\int_\Sigma VH_\Sigma f_\nu^2\,dA\right)\notag\\
&\le
\left(\int_\Sigma\frac{V}{H_\Sigma}\,dA\right)
\left(\frac{n}{n+1}A+c^2\mathcal K_\theta\right).
\label{eq:CS-final}
\end{align}
Since \eqref{eq:CS-final} holds for every $c\in\R$ and $B>0$ (because $V>0$ on $\overline\Omega$ and \(T\) has positive measure), it also shows that
\begin{equation*}
\mathcal K_\theta>0.
\end{equation*}
Indeed, if $\mathcal K_\theta<0$, the last factor on the right of \eqref{eq:CS-final} becomes negative for $|c|$ large. If $\mathcal K_\theta=0$, the right hand side stays bounded while $(A-cB)^2\to\infty$.

Dividing by the last factor in \eqref{eq:CS-final} gives
\begin{equation*}
\int_\Sigma\frac{V}{H_\Sigma}\,dA
\ge
\frac{(A-cB)^2}{\dfrac{n}{n+1}A+c^2\mathcal K_\theta}.
\end{equation*}
The right hand side attains its maximum at
\begin{equation*}
c_*=-\frac{n}{n+1}\frac{B}{\mathcal K_\theta}.
\end{equation*}
Substitution gives
\begin{equation*}
\int_\Sigma\frac{V}{H_\Sigma}\,dA
\ge
\frac{n+1}{n}A+\frac{B^2}{\mathcal K_\theta},
\end{equation*}
which is \eqref{eq:main-HK}.

Suppose now that equality holds.  Then equality must hold in \eqref{eq:trace-CS}. Since
$\operatorname{tr}\mathcal A_V[f]=1$, this gives \begin{equation}\label{eq:equality-tensor}
\mathcal A_V[f]=\frac1{n+1}g.
\end{equation}  On $\Sigma$, $f=0$ and all tangential derivatives of $f$ vanish.  For $X,Y\in T\Sigma$,
\[
\bar\nabla^2f(X,Y)=f_\nu h^\Sigma(X,Y).
\]
Restricting \eqref{eq:equality-tensor} to $T\Sigma$ yields
\begin{equation*}
f_\nu h^\Sigma=\frac1{n+1}g_\Sigma.
\end{equation*}
Thus $f_\nu\neq0$ and $\Sigma$ is totally umbilical. 
\end{proof}

\begin{remark}\label{rem:HK-equality}
Equality in Theorem~\ref{thm:substatic-HK} also forces the nonnegative error terms
in the weighted Reilly formula to vanish, so
\[
\int_\Omega\mathcal Q_V\bigl(W_V[f],W_V[f]\bigr)\,d\Omega=0,
\qquad
\int_T V\mathcal B_V^T(Z,Z)\,dA_T=0,
\]
where $Z=\nabla^T f-(f/V)\nabla^T V$.
The same conclusions hold in the free boundary case with $c=0$.
\end{remark}

The inequality also holds in the free boundary case. When the ambient
space admits a reflection across the support that preserves $g$ and $V$,
it also follows from the corresponding Heintze--Karcher inequality by doubling.

\begin{corollary}\label{cor:free-boundary}
Under the hypotheses of Theorem~\ref{thm:substatic-HK}, with the contact
angle replaced by $\theta=\pi/2$, we have
\begin{equation}\label{eq:free-HK}
\int_\Sigma\frac{V}{H_\Sigma}\,dA
\ge
\frac{n+1}{n}\int_\Omega V\,d\Omega.
\end{equation}
If equality holds, then $\Sigma$ is totally umbilical.
\end{corollary}

\begin{proof}
Take $c=0$ in \eqref{eq:mixed-problem}.  Then the contribution from $T$ is nonnegative by $V$-convexity, while the argument leading to \eqref{eq:key-Reilly-estimate} gives
\[
\int_\Sigma VH_\Sigma f_\nu^2\,dA\le\frac{n}{n+1}A.
\]
The Green identity gives $\int_\Sigma Vf_\nu=A$, and the same Cauchy-Schwarz argument yields \eqref{eq:free-HK}.  The equality argument is unchanged.
\end{proof}

To obtain an Alexandrov-type theorem, we use a conformal Killing field to derive a
Minkowski identity, as in \cite{WangXia2019,JiaXiaZhang2023}.
A vector field $X$ is conformal Killing with conformal factor $V$ if
\[
\frac12\mathcal L_Xg=Vg,
\]
or equivalently,
$g(\bar\nabla_U X,W)+g(\bar\nabla_W X,U)=2Vg(U,W)$ for all vector
fields $U,W$.
\begin{proposition}\label{prop:general-Minkowski}
Let $X$ be a conformal Killing field on a neighbourhood of $\overline\Omega$
such that $\frac12\mathcal L_Xg=Vg$. Then
\begin{equation}\label{eq:general-Minkowski}
\int_\Sigma\bigl(nV-H_\Sigma g(X,\nu)\bigr)\,dA
=\int_\Gamma g(X,\mu)\,ds,
\end{equation}
and
\begin{equation}\label{eq:general-X-flux}
\int_\Sigma g(X,\nu)\,dA
=(n+1)\int_\Omega V\,d\Omega-\int_Tg(X,N)\,dA_T.
\end{equation}
If $X$ is tangent to $T$, the right hand side of
\eqref{eq:general-Minkowski} is
$\cos\theta\int_\Gamma g(X,\bar\nu)\,ds$.
\end{proposition}

\begin{proof}
With our convention for the outward normal, we have
$\operatorname{div}_\Sigma X^\top=nV-H_\Sigma g(X,\nu)$.
Integration on $\Sigma$ proves \eqref{eq:general-Minkowski}.
Since $\operatorname{div}X=(n+1)V$, the divergence theorem on $\Omega$
gives \eqref{eq:general-X-flux}. The last assertion follows from
\eqref{eq:capillarity}.
\end{proof}

This gives the following Alexandrov-type theorem.

\begin{corollary}\label{cor:substatic-Alexandrov}
Assume the hypotheses of Theorem~\ref{thm:substatic-HK} when
$0<\theta<\pi/2$, or those of Corollary~\ref{cor:free-boundary} when
$\theta=\pi/2$. Suppose that $H_\Sigma=h_0>0$ is constant and that a
conformal Killing field $X$ satisfies
\[
\frac12\mathcal L_Xg=Vg,\qquad g(X,N)=0\quad\text{on }T.
\]
If $0<\theta<\pi/2$, assume in addition that there is a Killing field
$Y$ with
\[
\mathcal L_Yg=0,\qquad g(Y,N)=V\quad\text{on }T.
\]
The fields are defined on a neighbourhood of $\overline\Omega$.
Then $\Sigma$ is totally umbilical.
\end{corollary}

\begin{proof}
Since $X$ is tangent to $T$, we have $\operatorname{div}_TX=nV$ and
\[
\int_\Gamma g(X,\mu)\,ds
=\cos\theta\int_\Gamma g(X,\bar\nu)\,ds
=n\cos\theta B.
\]
Proposition~\ref{prop:general-Minkowski} therefore gives
\begin{equation}\label{eq:CMC-Minkowski-tangent}
\int_\Sigma\frac{V}{h_0}\,dA
=\frac{n+1}{n}A+\frac{\cos\theta}{h_0}B.
\end{equation}
For $\theta=\pi/2$, this is equality in \eqref{eq:free-HK}, and
Corollary~\ref{cor:free-boundary} proves the assertion.

Suppose that $0<\theta<\pi/2$. Since $Y$ is Killing and $g(Y,N)=V$,
the divergence theorem on $\Omega$ gives
\[
\int_\Sigma g(Y,\nu)\,dA=-B.
\]
On $\Sigma$ and $T$, respectively, we have
\[
\operatorname{div}_\Sigma Y^\top=-h_0g(Y,\nu),
\qquad
\operatorname{div}_T Y^\top=-H_TV.
\]
Using the capillary relations and integrating, we obtain
\begin{align*}
h_0B
&=\int_\Gamma g(Y,\mu)\,ds\\
&=\sin\theta C+\cos\theta\int_\Gamma g(Y,\bar\nu)\,ds\\
&=\sin\theta C-\cos\theta\int_T H_TV\,dA_T
=\cos\theta\,\mathcal K_\theta.
\end{align*}
Substituting into \eqref{eq:CMC-Minkowski-tangent} gives
\[
\int_\Sigma\frac{V}{h_0}\,dA
=\frac{n+1}{n}A+\frac{B^2}{\mathcal K_\theta}.
\]
Thus equality holds in \eqref{eq:main-HK}, and
Theorem~\ref{thm:substatic-HK} proves the assertion.
\end{proof}

As an application, we obtain a Heintze--Karcher inequality for hypersurfaces with constant
mean curvature.

\begin{lemma}\label{lem:transverse-Killing}
Let $0<\theta<\pi/2$. Suppose that $Y$ is a Killing field on a neighbourhood of
$\overline\Omega$ and that $g(Y,N)=V$ on $T$. Then
\begin{equation}\label{eq:transverse-flux}
\int_\Sigma g(Y,\nu)\,dA=-B,
\qquad
\int_\Gamma g(Y,\mu)\,ds=\cos\theta\,\mathcal K_\theta,
\end{equation}
and
\begin{equation}\label{eq:transverse-H-flux}
\int_\Sigma H_\Sigma g(Y,\nu)\,dA
=-\cos\theta\,\mathcal K_\theta.
\end{equation}
\end{lemma}

\begin{proof}
The first identity follows from $\operatorname{div}Y=0$.
Since $Y$ is Killing,
$\operatorname{div}_T Y^\top=-H_Tg(Y,N)=-H_TV$, and hence
\[
\int_\Gamma g(Y,\bar\nu)\,ds=-\int_TH_TV\,dA_T.
\]
The capillary relation gives the second identity in
\eqref{eq:transverse-flux}. Finally,
$\operatorname{div}_\Sigma Y^\top=-H_\Sigma g(Y,\nu)$ gives
\eqref{eq:transverse-H-flux}.
\end{proof}

This gives the following inequality.

\begin{corollary}
\label{cor:static-CMC-HK}
Under the hypotheses of Theorem~\ref{thm:substatic-HK}, suppose that
$(\overline M,g,V)$ is static and that $H_\Sigma=h_0>0$ is constant.
Assume that there is another static potential $U$, not necessarily positive,
such that
\begin{equation*}
\mathcal Q_U=0,
\qquad U_N-\frac{V_N}{V}U=1\quad\hbox{on }T.
\end{equation*}
Then
\begin{equation}\label{eq:static-CMC-HK}
\int_\Sigma
\frac{V+\cos\theta\,(VU_\nu-UV_\nu)}{H_\Sigma}\,dA
\geq\frac{n+1}{n}\int_\Omega V\,d\Omega.
\end{equation}
If equality holds, then $\Sigma$ is totally umbilical.
The free boundary version follows from Corollary~\ref{cor:free-boundary}
without requiring $U$.
\end{corollary}

\begin{proof}
Set $Y=V\bar\nabla U-U\bar\nabla V$. Taking the trace of the static
equations gives $n\bar\Delta U+\mathrm{Scal}_gU=0$ and
$n\bar\Delta V+\mathrm{Scal}_gV=0$. Therefore
\[
\tfrac12\mathcal L_Yg
=V\bar\nabla^2U-U\bar\nabla^2V=0,
\qquad g(Y,N)=V.
\]
Lemma~\ref{lem:transverse-Killing} and the constancy of $H_\Sigma$ give
$\cos\theta\,\mathcal K_\theta=h_0B$. Thus the correction in
\eqref{eq:main-HK} is
\[
\frac{B^2}{\mathcal K_\theta}
=\frac{\cos\theta}{h_0}B
=-\int_\Sigma\frac{\cos\theta\,g(Y,\nu)}{H_\Sigma}\,dA.
\]
Moving this term to the left proves \eqref{eq:static-CMC-HK}.
\end{proof}

The form of \eqref{eq:static-CMC-HK} is closer to the classical
Heintze--Karcher inequality and its weighted version in \cite{Brendle2013}.
 The inequality \eqref{eq:static-CMC-HK} in the ambient manifolds that admit these
fields is expected to be true even without assuming a constant mean curvature. This has been established
in the hyperbolic half-space in \cite{HuWeiXiaZhou2025} (even for \(\theta> \frac{\pi}{2}\)).

We now specialize these results to the round sphere,
completing the space form cases of the capillary Heintze--Karcher inequality.

\subsection{The spherical case}\label{sec:spherical-HK}

Fix $O\in\Sph^{n+1}$ and a unit vector $a\in T_O\Sph^{n+1}$.
In geodesic polar coordinates $q=\exp_O(r\omega)$, put
\begin{equation*}
V_0=\cos r,
\qquad V_a=\sin r\,g_O(a,\omega).
\end{equation*}
Both potentials extend smoothly across $O$ and satisfy
$\bar\nabla^2V=-Vg$, so $\mathcal Q_{V_0}=\mathcal Q_{V_a}=0$.
For $R\in(0,\pi/2)$, let $S_R=\{r=R\}$ and
$B^+_{R,a}=\{r<R,\ V_a>0\}$. The corresponding hyperbolic choices
recover the results in \cite[Section~6]{ChenPyo2026}, as explained in
Remark~\ref{rem:hyperbolic-application}. With $N=\partial_r$, we have
\begin{equation*}
h^{S_R}=\cot R\,g_{S_R},\qquad
(V_a)_N=\cot R\,V_a,\qquad \mathcal B_{V_a}^{S_R}=0.
\end{equation*}

Define
\begin{equation*}
Y_a=V_0\bar\nabla V_a-V_a\bar\nabla V_0,\qquad
X_a=\cos R\,Y_a-\bar\nabla V_a,\qquad U_a=-\sin R\,V_0.
\end{equation*}
The static equations give
\begin{equation}\label{eq:spherical-field-identities}
\frac12\mathcal L_{X_a}g=V_ag,\qquad
\mathcal L_{Y_a}g=0,\qquad
V_a\bar\nabla U_a-U_a\bar\nabla V_a=\sin R\,Y_a.
\end{equation}
On $S_R$, these fields and potentials satisfy
\begin{equation}\label{eq:spherical-boundary-fields}
g(X_a,N)=0,\qquad
g(\sin R\,Y_a,N)=V_a,\qquad
(U_a)_N-\frac{(V_a)_N}{V_a}U_a=1.
\end{equation}
Let $\Omega$ be a domain with $\partial\Omega=\Sigma\cup T$,
$T\subset S_R$, $\Sigma\setminus\Gamma\subset B^+_{R,a}$ and
$V_a>0$ on $\overline\Omega$, with the notation and smoothness assumptions
of Section~\ref{sec:preliminaries}. Suppose that $\Sigma$ meets $T$ at
a constant angle $\theta\in(0,\pi/2]$. For $0<\theta<\pi/2$,
Lemma~\ref{lem:transverse-Killing} gives
\begin{equation*}
\begin{split}
D&:=\sin R\int_\Gamma g(Y_a,\mu)\,ds
=\cos\theta\,\mathcal K_\theta,\\
\mathcal K_\theta
&=\tan\theta\int_\Gamma V_a\,ds
-n\cot R\int_T V_a\,dA_T.
\end{split}
\end{equation*}
Applying Theorem~\ref{thm:substatic-HK} and
Corollary~\ref{cor:free-boundary} gives the following.

\begin{corollary}\label{cor:sphere}
Under the preceding assumptions, if $H_\Sigma>0$, then
\begin{equation}\label{eq:spherical-final}
\int_\Sigma\frac{V_a}{H_\Sigma}\,dA
\geq\frac{n+1}{n}\int_\Omega V_a\,d\Omega
+\cos\theta\,
\frac{\left(\int_T V_a\,dA_T\right)^2}
{\sin R\int_\Gamma g(Y_a,\mu)\,ds}.
\end{equation}
For $\theta=\pi/2$, the last term is zero.
If equality holds, then $\Sigma$ is totally umbilical. 
\end{corollary}

Proposition~\ref{prop:general-Minkowski} and the flux identity
$\sin R\int_\Sigma g(Y_a,\nu)\,dA=-\int_T V_a\,dA_T$
give the spherical Minkowski formula
\begin{equation*}
\int_\Sigma
\bigl(nV_a+n\sin R\cos\theta\,g(Y_a,\nu)
-H_\Sigma g(X_a,\nu)\bigr)\,dA=0.
\end{equation*}
When $H_\Sigma$ is constant, Corollary~\ref{cor:static-CMC-HK} applied
with $U=U_a$ also gives
\begin{equation}\label{eq:spherical-CMC-HK}
\int_\Sigma
\frac{V_a+\sin R\cos\theta\,g(Y_a,\nu)}{H_\Sigma}\,dA
\geq\frac{n+1}{n}\int_\Omega V_a\,d\Omega.
\end{equation}
By \eqref{eq:spherical-field-identities} and
\eqref{eq:spherical-boundary-fields},
Corollary~\ref{cor:substatic-Alexandrov} also applies.

\begin{corollary}\label{thm:spherical-Alexandrov}
Under the assumptions of Corollary~\ref{cor:sphere}, suppose that
$n\geq2$, $\Sigma$ is connected and $H_\Sigma$ is constant.
Then $\Sigma$ is totally umbilical and is contained in a geodesic sphere.
\end{corollary}

\begin{remark}\label{rem:hyperbolic-application}
In $\HH^{n+1}$, take instead
\[
V_0=\cosh r,\qquad V_a=\sinh r\,g_O(a,\omega),\qquad
Y_a=V_0\bar\nabla V_a-V_a\bar\nabla V_0,
\]
and, for $R>0$, set
\[
X_a=\bar\nabla V_a-\cosh R\,Y_a,\qquad U_a=-\sinh R\,V_0.
\]
Now $\bar\nabla^2V=Vg$ and $\mathcal Q_V=0$. On $S_R=\{r=R\}$,
we have $h^{S_R}=\coth R\,g_{S_R}$ and
$(V_a)_N=\coth R\,V_a$. The field identities
\eqref{eq:spherical-field-identities} and
\eqref{eq:spherical-boundary-fields} hold with $\sin R$ replaced by
$\sinh R$. Hence the same applications, on domains contained in
$\{r\leq R,\ V_a>0\}$, recover the capillary Heintze--Karcher inequality
and the CMC umbilicity conclusion of \cite[Section~6]{ChenPyo2026}.
The formulas \eqref{eq:spherical-final} and \eqref{eq:spherical-CMC-HK}
hold with these potentials and fields and with $\sin R$ replaced by
$\sinh R$.
\end{remark}

\section*{Acknowledgements}
The author was partially supported by the Brazilian National Council for
Scientific and Technological Development (CNPq), grants 409513/2023-7 and
406078/2025-4.

\section*{AI disclosure}
Large language models were used for copy-editing and for proofreading the
exposition. They were also used to suggest the notation adopted in the proof
of Theorem~\ref{thm:substatic-HK}, which considerably improved its
presentation.

\printbibliography

\end{document}